\documentclass[12pt, a4paper]{amsart}
\usepackage[utf8]{inputenc}
\usepackage{amsfonts,amssymb,amsxtra,amsthm,amsmath,amscd,mathrsfs,cite}
\usepackage[all]{xy}

\usepackage{tikz}
\usepackage{verbatim}
\usepackage[normalem]{ulem}
\usepackage{soul}
\usepackage{color}

\setstcolor{red}

\makeatletter
\@namedef{subjclassname@2010}{%
	\textup{2010} Mathematics Subject Classification}
\makeatother

\def\leq{\leqslant}
\def\geq{\geqslant}
\def\le{\leqslant}
\def\ge{\geqslant}

\newcommand{\f}{\frac}

\theoremstyle{plain}
\newtheorem{theorem}{Theorem}[section]
\newtheorem{proposition}{Proposition}[section]
\newtheorem{lemma}[proposition]{Lemma}

\theoremstyle{remark}

\numberwithin{equation}{section}

\addtocounter{footnote}{1}

\begin{document}
	
	\title[Anisotropic quadratic equations in three variables]  
	{Anisotropic quadratic equations \\ in three variables\\ 
		{\tiny\rm To the memory of John Coates}} 
	
	\author{Jiamin Li \& Jianya Liu}
	
	\address{%
		Jiamin Li
		\\
		Mathematical Research Center
		\\
		Shandong University
		\\
		Jinan, Shandong 250100
		\\
		P. R. China
	}
	\email{lijiamin@mail.sdu.edu.cn}
	\address{%
		Jianya Liu
		\\
		School of Mathematics \& Mathematical Research Center 
		\\
		Shandong University
		\\
		Jinan, Shandong 250100
		\\
		P. R. China
	}
	\email{jyliu@sdu.edu.cn}

\date{\today}
	
\begin{abstract}
		Let $f(x_1, x_2, x_3)$ be an indefinite anisotropic integral quadratic form with determinant $d(f)$, 
		and $t$ a non-zero integer such that $d(f)t$ is square-free. It is proved in this paper that, 
		as long as there is one integral solution to $f(x_1, x_2, x_3) = t$,  
		there are infinitely many such solutions for which (i) $x_1$ has at most $6$ prime factors, and (ii) the product $x_1 x_2$ 
		has at most $16$ prime factors. Various methods, such as algebraic theory of quadratic forms, harmonic analysis, Jacquet-Langlands  theory, 
		as well as combinatorics, interact here, and the above results come from applying the sharpest known bounds towards 
		Selberg’s eigenvalue conjecture. Assuming the latter the number $6$ or $16$ may be reduced to $5$ or $14$, respectively. 
\end{abstract}
	

	\maketitle
	
	\section{Introduction}
	
	\subsection{Quadratic equations in primes}  
	Let $f({\bf x})=f(x_1, x_2, \ldots, x_s)$ be an integral indefinite quadratic form in $s\geq 3$ variables, and $t$ an integer, 
	where here and throughout bold face letters denote vectors whose dimensions are clear from the context. 
	The question whether the equation 
	\begin{equation}\label{fxs=t} 
	f(x_1, x_2, \ldots, x_s)=t 
	\end{equation} 
	has prime solutions is hard in general. The assumption 
	that $f$ is indefinite is just for the ease of presentation; it is obvious how to modify the statements in this paper for positive/negative 
	definite quadratic forms $f$. In this direction, Hua \cite{Hua} proved in 1938 that if $f$ is diagonal 
	and $s\geq 5$ then \eqref{fxs=t} is solvable in primes, provided there is no local obstructions. 
	For general $f$ that may have cross terms, it is proved much later in \cite{Liu} that \eqref{fxs=t} is solvable in primes 
	provided that $s\geq 10$ and there is no local obstructions. The above bound $10$ has subsequently been 
	reduced to $9$ by Zhao \cite{Zha}, and then to $8$ by Green \cite{Gre}.  
	
	As for the least number of variables needed so that \eqref{fxs=t} is solvable in primes, the optimal 
	conjecture is $s=3$; thus \eqref{fxs=t} reduces to 
	\begin{equation}\label{fx3=t}  
	f(x_1, x_2, x_3)=t.  
	\end{equation} 
	Under the general local-global principle of Bourgain, Gamburd, and Sarnak \cite{BouGamSar}, 
	it is conjectured in \cite{LS10} that \eqref{fx3=t} should be solvable in primes 
	provided there is no local obstructions. 
	This conjecture seems, however, much harder than the twin prime conjecture, and therefore at the present stage of knowledge one can only hope to get almost-prime solutions to \eqref{fx3=t}.

	For $f({\bf x})=f(x_1,x_2,x_3)$ as in \eqref{fx3=t} and $t$ a non-zero integer, we denote the affine quadric
	$$
	V=V_{f,t} =\{{\bf x}: f({\bf x})=t\}, 
	$$
	and denote by $V(\Bbb Z)$ the set of integer points on $V$, that is 
	$$
	V(\Bbb Z)=\{{\bf x}\in \Bbb Z^3: f({\bf x})=t\}. 
	$$ 
	For $d\geq 1$,
	let
	\begin{equation} 
	V(\Bbb Z/d\Bbb Z)=\{{\bf x}\in (\Bbb Z/d\Bbb Z)^3: f({\bf x})\equiv t(\bmod\ d)\},
	\end{equation} 
	and
	\begin{equation} 
	V^0(\Bbb Z/d\Bbb Z)=\{{\bf x}\in V(\Bbb Z/d\Bbb Z): x_1x_2x_3\equiv 0(\bmod\ d)\}.
	\end{equation} 
	A prime $p$ is {\it bad} or {\it ramified} if 
	\begin{equation} 
	V^0(\Bbb Z/p\Bbb Z)=V(\Bbb Z/p\Bbb Z).
	\end{equation} 
	It is proved in \cite{LS10} that these bad primes are in the set 
	\begin{equation}\label{def/B} 
	B:=\{2,3,5,7\}. 
	\end{equation} 
	We say $x\in P_{r}(B)$ if $x$ has at most $r$ prime factors outside $B$, and we simply 
	write $P_{r}$ instead of $P_{r}(B)$ if there is no such a set $B$. 
	
	\medskip 
	
	Before going further, we must distinguish two cases: $f$ isotropic or anisotropic. 
	For isotropic $f$, there are non-constant polynomial parameterizations
	of points in $V(\Bbb Z)$. That is, there are morphisms $Q=(Q_1, Q_2, Q_3)$ with each $Q_j$
	an integral polynomial in one variable for which $Q(y)\in V(\Bbb Z)$
	if $y\in \Bbb Z$. Treatment of the anisotropic and isotropic cases are different, and will be summarized in 
	\S1.2 and \S1.3, respectively. 
	The purpose of this paper is to investigate almost-prime solutions to \eqref{fx3=t} for anisotropic 
	$f$. 
	
	\subsection{The anisotropic case} 
	Assume that $f$ is anisotropic over $\Bbb Q$, that is 
	there are no non-constant integral polynomial morphisms of the line into $V$. 
	Assume that $t d(f)$ is square-free, where $d(f)$ is the determinant of $f$. 
	Then $V(\Bbb Z)\not=\emptyset$ if and only if $f({\bf x}) \equiv t(\bmod\ d)$ is solvable for every $d\geq 1$.  
	This is a consequence of Siegel's mass formula \cite{Sie} for $f$, as well as the fact that such an $f$ only has one class in 
	its genus \cite[page 203]{Cas}.  We assume in the following that $V(\Bbb Z)\not=\emptyset$.

	Blomer and Br\"udern \cite{BloBru} have investigated \eqref{fx3=t} for diagonal anisotropic $f$; 
	they show that \eqref{fx3=t} has solutions ${\bf x}$ with each $x_j \in P_{521}$ for $t$ large and 
	satisfying suitable congruence conditions. 
	
	The main results of this paper are the following two theorems, in which just one or two variables are required to 
	be almost-primes. 
	
	\begin{theorem}\label{Thm1} 
		Let $f$ be an integral indefinite anisotropic quadratic form in $3$ variables, and $t$ a non-zero integer. 
		Assume that $t d(f)$ is square-free, where $d(f)$ is the determinant of $f$. 
		Let $B$ be as in \eqref{def/B}. Then the set 
		$$ 
		\{{\bf x} \in V(\mathbb{Z}): x_1\in P_{6}(B)\} 
		$$ 
		is Zariski dense in $V$. Assuming Selberg's eigenvalue conjecture, the above $6$ can be reduced to $5$. 
	\end{theorem} 
	
	We need a few words to clarify the meaning of Selberg's eigenvalue conjecture that appears in the 
	statement of above theorem. Let $\Gamma$ be a co-compact lattice in $SL_2(\Bbb R)$, and $\Gamma(q)$ 
	the principal congruence subgroup of $\Gamma$ of
	level $q\geq 1$.  Let $\Bbb H$ be the upper half-plane, and $\lambda_1$ the first eigenvalue of the Laplacian
	on the hyperbolic surface $\Gamma(q)\backslash \Bbb H$. 
	Selberg's eigenvalue conjecture \cite{Sel} 
	states that 
	$$
	\lambda_1\left(\Gamma(q)\backslash \Bbb H\right)\geq \f{1}{4}. 
	$$
	There are non-trivial bounds of the form 
	$$
	\lambda_1\left(\Gamma(q)\backslash \Bbb H\right)\geq \f{1}{4}-\theta^2 
	$$
	towards the conjecture, and the best result to-date is $\theta=7/64$ due to Kim and Sarnak 
	\cite{KimSar}. See \S2 on how these bounds are used in this paper.  
	
	\begin{theorem}\label{Thm2} 
		Let $f$ be an integral indefinite anisotropic quadratic form in $3$ variables, and $t$ a non-zero integer. 
		Assume that $t d(f)$ is square-free, where $d(f)$ is the determinant of $f$. Let $B$ be as in \eqref{def/B}. 
		Then the set 
		$$ 
		\{{\bf x} \in V(\mathbb{Z}): x_1x_2\in P_{16}(B)\} 
		$$ 
		is Zariski dense in $V$. On Selberg's eigenvalue conjecture, the above $16$ can be replaced by $14$. 
	\end{theorem}
	
	The above theorems can be viewed as continuation of the research in \cite{LS10} where all the three variables are  
	almost-primes. To be more precise, the main theorem in \cite{LS10} states that the set 
	$$ 
	\{{\bf x} \in V(\mathbb{Z}): x_1x_2x_3\in P_{26}(B)\} 
	$$ 
	is Zariski dense in $V$, and assuming Selberg's eigenvalue conjecture, one has $22$ in place of $26$. 
	
	We follow the framework in \cite{LS10}, where interact various methods, such as algebraic theory of quadratic forms, 
	harmonic analysis, Jacquet-Langlands theory, as well as combinatorics. We define an archimedean height to order the points, and we sieve
	on the orbit on the homogeneous space $V$. 
	The groups appear in the analysis are congruence
	subgroups of the spin group, and therefore non-trivial bounds 
	towards Selberg's eigenvalue conjecture are crucial. The dimensions of 
	the sieves in Theorems \ref{Thm1} and \ref{Thm2} are $1$ and $2$, respectively.  
	Both Theorems \ref{Thm1} and \ref{Thm2} are proved finally by the weighted sieves of 
	Diamond and Halberstam \cite{DiaHal} or \cite{DiaHal08}. 
	To complete the non-linear sieves in Theorems \ref{Thm2}, 
	we need Halberstam and Richert \cite[Chapter 10]{HR74}, in addition to the above two papers.  
	
	\subsection{The isotropic case}  
	The isotropic case can be treated differently since there are non-constant polynomial parameterizations
	of points in $V(\Bbb Z)$. In the case of isotropic $f({\bf x})=x_1^2+x_2^2-x_3^2$ and $t=0$, 
	the Pythagorean equation 
	\begin{equation}\label{Pyt} 
	x_1^2+x_2^2-x_3^2=0
	\end{equation} 
	has parametrized solutions 
	\begin{equation}\label{Pyt/Sol} 
	x_1=m^2-n^2, \ 
	x_2=2mn, \ 
	x_3=m^2+n^2 
	\end{equation} 
	where $m, n\in \Bbb Z$. Diamond and Halberstam \cite{DiaHal} show that there are infinitely many
	Pythagorean triples for which $x_1x_2x_3\in P_{17}$. In their proof, they fix $n=1$ and 
	apply a $2$-dimensional sieve to $m$ 
	in \eqref{Pyt/Sol}, and hence their solutions are not Zariski dense in $V$.  
	Using similar arguments, one can do the same for $V_{f,t}$ when $f$ is isotropic. 
	
	Let $G = SO_f(\Bbb R)$ be the real special orthogonal group of $3\times 3$ matrices preserving the Pythagorean quadratic form
	$$
	f({\bf x})= x_1^2 + x_2^2 - x_3^2. 
	$$ 
	Let $\Gamma < G(\Bbb Z)$ be a finitely generated subgroup of the integer matrices in $G$, 
	and assume $\Gamma$ is non-elementary, or equivalently, that its Zariski closure is $SO_f$. 
	For a fixed Pythagorean triple, for example ${\bf y} = (3, 4, 5)$, we form the orbit
	$$
	\mathscr{O}={\bf y}\Gamma. 
	$$
	The group $\Gamma$ and hence the orbit $\mathscr{O}$ are allowed to be thin, but not too thin. Put 
	\begin{eqnarray}
	\left\{
	\begin{array}{lll}
	g_1({\bf x})=x_3, \\
	g_2({\bf x})=\frac{1}{12}x_1x_2, \\
	g_3({\bf x})=\frac{1}{60}x_1x_2x_3. 
	\end{array}
	\right.
	\end{eqnarray}
	For each $j=1, 2, 3,$ one wants to find $r_j$ as small as possible such that the set 
	$$
	\{{\bf x}\in \mathscr{O}: g_j({\bf x})\in P_{r_j}\} 
	$$ 
	is Zariski dense in the Zariski closure of $\mathscr{O}$. This has been done  
	in a series of papers such as Kontorovich \cite{Kon09}, 
	Kontorovich and Oh \cite{KO12}, 
	Bourgain and Kontorovich \cite{BK15}, 
	Hong and Kontorovich \cite{KH15}, and Ehrman \cite{E19}. 
	In particular the following results have been achieved: 
	$r_1=4$ by \cite{BK15} and \cite{E19}; $r_2=19$ and $r_3=28$ by \cite{E19}. 
	
	It is interesting to compare these with Theorems \ref{Thm1} and \ref{Thm2} above, as well as that 
	in \cite{LS10}.  
	
	\section{Equidistribution modulo $d$} 
	
	The main task of this section is to prove an equi-distribution theorem for one of the variables in the 
	framework of \cite{LS10}, and we note that the situation for two variables is similar. 
	
	\subsection{The orbits}
	Let $f$ and $t$ be as in Theorem~\ref{Thm1}. Denote by $SO_f$ the special 
	orthogonal group of $3\times 3$ matrices which preserve $f$. It is a linear algebraic group defined over $\Bbb Q$.
	As in Cassels \cite{Cas}, denote by $G$ the spin double cover of $SO_f$. Then $G$ is defined over $\Bbb Q$ and, since $f$ is anisotropic, $G$ consists of
	the elements of norm $1$ in a quaternion division algebra $D_f$ over $\Bbb Q$. Also there exists a morphism $\tau: G\to SO_f$; the reader is referred to \cite[\S2]{LS10} for an exact definition of $\tau$. 
	
	Let $\Gamma$ be the unit group of integral quaternions. For simplicity we may drop 
	$\tau$ and write the action of $\gamma\in \Gamma$ on ${\bf x}=(x_1,x_2,x_3)$
	by 
	$$
	{\bf x}\mapsto {\bf x}\gamma. 
	$$ 
	Now $\tau(\Gamma)\subset SO_f(\Bbb Z)$ and $V(\Bbb Z)$
	decomposes into finitely many $\Gamma$ orbits
	\begin{eqnarray}
	V(\Bbb Z)=\bigsqcup_{j=1}^h {\bf y}^{(j)}\Gamma.
	\end{eqnarray}
	Note that the quadric $V(\Bbb Z)$ in $3$-variables does not satisfy strong approximation; 
	the reader is referred to Borovoi \cite{Bor}.  
	So we will conduct our analysis on each orbit ${\bf y}^{(j)}\Gamma$ separately, and we denote by 
	\begin{eqnarray}\label{def/OO}
	\mathscr{O}={\bf y}\Gamma 
	\end{eqnarray}
	for such a fixed orbit. 
	
	
	\medskip 
	
	The analysis on \eqref{def/OO} is as follows. 
	Let $K$ be a maximal compact subgroup of $SO_f(\Bbb R)$, and let $|\cdot |$ be a 
	Euclidean norm on $\Bbb R^3$ which is $K$ invariant. To do the counting on $\mathscr{O}$, 
	we also need good weight functions, and these have already been constructed in \cite{LS10}, 
	which are smooth functions $F_T: {\Bbb R}^3\to \Bbb R$ for $T\geq 10$ that 
	depend only on $|{\bf x}|$ and satisfy
	\begin{eqnarray}
	\left\{
	\begin{array}{lll}
	0\leq F_T({\bf x})\leq 1, \\
	F_T({\bf x})=1, & \mbox{if } |{\bf x}|\leq T/c_0,  \\
	F_T({\bf x})=0, & \mbox{if } |{\bf x}|\geq c_0 T.  
	\end{array}
	\right.
	\end{eqnarray}
	Here $c_0$ is a positive constant depending only on $\mathscr{O}$. 
	With these weight $F_T$ we can define, for $n\geq 0$, 
	\begin{eqnarray}\label{anT}
	a_n(T)=\sum_{{\bf x}\in \mathscr{O} \atop x_1=\pm n}F_T({\bf x}).
	\end{eqnarray}
	Plainly $a_n(T)\geq 0$, and $a_n(T)=0$ if $n\geq c_0 T$. We remark that the $a_n$ here 
	has different meaning as that in \cite{LS10}. 
	
	\begin{lemma}\label{a0tUpper}
		We have 
		\begin{eqnarray*}\label{a0T}
			a_0(T)\ll_{f}\log (c_{0}T).
		\end{eqnarray*}
	\end{lemma}
	
	\begin{proof}
		The definition in \eqref{anT} gives 
		\begin{equation*}
		a_0(T)=\sum_{{\bf x}\in \mathscr{O} \atop x_1=0}F_T({\bf x}),
		\end{equation*}
		and the latter equals $S$, the weighted number of integral solutions to the equation 
		$f(0,x_2,x_3)=t$ in the region $\sqrt{x_2^2+x_3^2}\le c_0T$. If we write 
		$$
		f(0,x_2,x_3)=ax_2^2+bx_2x_3+cx_3^2 
		$$
		and $\delta=b^2-4ac,$ 
		then the argument in \cite[Lemma 7.1]{LS10} gives 
		$$
		S\ll \log T+\log(|a|+|b|+|\delta|)c_0, 
		$$ 
		where the implied constant is absolute. The lemma is proved. 
	\end{proof}
	
\begin{lemma}\label{antUpper}
Write 
		\begin{eqnarray}\label{def/X}
		X=\sum\limits_{n\geq 1} a_n(T).
		\end{eqnarray}
		Then, as $T\to \infty$, 
		 \begin{eqnarray}
		X\asymp T, 
		\end{eqnarray}
where $X\asymp T$ means $T\ll X\ll T$. 
\end{lemma}
	
	\begin{proof}
		We note that $X+a_{0}(T)$ counts the weighted number of solutions on $\mathscr{O}$, that is  
		$$
		\sum_{{\bf x}\in \mathscr{O}}F_T({\bf x}). 
		$$
		On the other hand, the above quantity also equals $Y+a_{0}^{*}(T)$ with 
		$$
		Y=\sum\limits_{n\ge 1}\sum\limits_{{\bf x}\in \mathscr{O} \atop x_1x_2x_3=\pm n}F_T({\bf x}),
		\quad
		a_0^{*}(T)=\sum_{{\bf x}\in \mathscr{O} \atop x_1x_2x_3=0}F_T({\bf x}). 
		$$
		It follows that 		
		$$
		X=Y+a_{0}^{*}(T)-a_{0}(T). 
		$$ 
		By \cite[(3.4),(3.60) and(3.63)]{LS10} we have $Y\asymp T$ as $T\to \infty$. 
		The upper bound $a_{0}^{*}(T)\ll_{f}\log (c_{0}T)$ is proved in \cite[Lemma 7.1]{LS10}. These together 
		with Lemma \ref{a0tUpper} give the desired estimate for $X$.  
	\end{proof}
	
	\subsection{An equi-distribution result} 
	Note that $D_f$ is splitting at $\infty$, and this yields 
	$$
	D_f\otimes \Bbb R \cong M_2(\Bbb R). 
	$$
	This realizes $\Gamma$ as a co-compact lattice in $SL_2(\Bbb R)$. 
	Let $\Gamma(q)$ be the principal congruence subgroup of $\Gamma$ of
	level $q\geq 1$, and $\lambda_1$ the first eigenvalue of the Laplacian
	on the hyperbolic surface $X_{\Gamma(q)}=\Gamma(q)\backslash \Bbb H$. Then 
	there is $0\leq \theta <1/4$ so that 
	\begin{eqnarray}
	\lambda_1(X_{\Gamma(q)})\geq \f{1}{4}-\theta^2.
	\end{eqnarray}
	By the Jacquet-Langlands correspondence \cite{JacLan} and the bound of 
	Kim-Sarnak \cite{KimSar}, one sees that $\theta=7/64$ is admissible,  
	while Selberg's eigenvalue conjecture \cite{Sel} implies that $\theta=0$. 
	
	Denote by $\mathscr{O}(\Bbb Z/d\Bbb Z)$ the orbit ${\bf y}\Gamma$ in $(\Bbb Z/d\Bbb Z)^3$ 
	stated in the previous subsection, and by
	$\mathscr{O}^0(\Bbb Z/d\Bbb Z)$ the subset of $\mathscr{O}(\Bbb Z/d\Bbb Z)$ for which
	$y_1 \equiv 0(\bmod\ d)$. Then we  have the following equi-distribution result. 
	
	\begin{lemma}\label{fO/dens}
		Let $f$ and $\mathscr{O}$ be as above, and $X$ as in \eqref{def/X}. 
		Then, for $1\leq d\leq T$,
		we have
		\begin{equation}\label{anTd}
		\sum_{n\equiv 0(\bmod d)\atop n\geq 1} a_n(T)
		=\dfrac{\omega(d)}{d}X
		+ O_\varepsilon (d^{1+\varepsilon}T^{1/2+\theta+\varepsilon}),
		\end{equation}
		where 
		\begin{equation}\label{local}
		\dfrac{\omega(p)}{p}=\f{|\mathscr{O}^0(\Bbb Z/p\Bbb Z)|}{|\mathscr{O}(\Bbb Z/p\Bbb Z)|}
		=\frac{1}{p}+O\left(\frac{1}{p^2}\right). 
		\end{equation}
	\end{lemma}
	
	\begin{proof}
		Let $\mathscr{O}_{3}^0(\Bbb Z/d\Bbb Z)$ be the subset of $\mathscr{O}(\Bbb Z/d\Bbb Z)$ for which
		$y_1y_2y_3 \equiv 0(\bmod\ d)$. We have that 
		$\mathscr{O}^{0}(\Bbb Z/d\Bbb Z)\subset \mathscr{O}_{3}^0(\Bbb Z/d\Bbb Z)$ 
		since $y_1 \equiv 0(\bmod\ d)$ implies $y_1y_2y_3 \equiv 0(\bmod\ d)$.  
		Thus, for any square-free $d$, 
		$$
		|\mathscr{O}^{0}(\Bbb Z/d\Bbb Z)|\ll d^{1+\varepsilon}
		$$
		by \cite[(4.14)]{LS10}.  
		
		To analyze the local density \eqref{local}, we note that $d(f)t$ is square-free, and hence 
		for $p\nmid d(f)t$, Cassels\cite[Exercise 13 on page 31]{Cas} gives
		$$
		|V_t(\Bbb Z/p\Bbb Z)|=p^2+\bigg(\dfrac{-dt}{p}\bigg)p
		$$
		and
		$$
		|V_t^0(\Bbb Z/p\Bbb Z)|=|\{(x_2,x_3)\in\Bbb F_p^2: f(0,x_2,x_3)\equiv 0(\bmod\ p)\}|=p\pm 1. 
		$$
		It follows that 
		\eqref{local} holds by the same argument as in \cite[(4.8)-(4.13)]{LS10}.
		With Lemmas \ref{a0tUpper} and \ref{antUpper}, 
		the proof of \eqref{anTd} is exactly the same as that of \cite[Theorem 2.1]{LS10}.
	\end{proof}
	
	As with $V(\Bbb Z/p\Bbb Z)$ and $V^0(\Bbb Z/p\Bbb Z)$, we say that $p$
	is {\it bad} or {\it ramified} for $\mathscr{O}$ if 
	$\mathscr{O}(\Bbb Z/p\Bbb Z)=\mathscr{O}^0(\Bbb Z/p\Bbb Z)$ 
	where $\mathscr{O}^0(\Bbb Z/p\Bbb Z)$ is the subset of $\mathscr{O}(\Bbb Z/p\Bbb Z)$ for which
	$y_1 \equiv 0(\bmod\ p)$.
	It is proved in \cite[\S4]{LS10} that $\mathscr{O}_3 (\Bbb Z/p\Bbb Z)\not=\mathscr{O}_{3}^0(\Bbb Z/p\Bbb Z)$ when $p>7$, 
	and hence $\mathscr{O}(\Bbb Z/p\Bbb Z)\not=\mathscr{O}^0(\Bbb Z/p\Bbb Z)$ when $p>7$. The only possibility
	for a bad $p$ for $\mathscr{O}$ is a $p$ that lies in $B$ defined in \eqref{def/B}.  
	
	\section{Weighted sieves}
	
	In this section we introduce some results of weighted sieves, 
	which will be used to prove the main results of this paper. 
	
	\subsection{The nonlinear case $\kappa> 1$}
	Let $\mathscr{A}$ be a finite sequence of real numbers $a_n\ge 0$, and $B$ a fixed finite set of primes. We are interested in a reasonable lower-bound estimate for
	\begin{equation}\label{weight}
	\sum\limits_{n\in P_r(B)}a_n,
	\end{equation}
	where $P_r(B)$ is the set of positive integers with at most $r$ prime divisors outside $B$. To estimate \eqref{weight}, we need to know how $\mathscr{A}$ is distributed to each of the arithmetic progression $0(\bmod\ d)$, where $d$ is square-free and $(d, B)=1$. To this end, let $d$ be a square-free number, and write
	$$
	\mathscr{A}_d=\left\{a_n \in \mathscr{A}: n \equiv 0(\bmod\ d)\right\} .
	$$
	Suppose there exists an approximation $X$ to 
	$|\mathscr{A}|:=\sum a_n$ 
	and a non-negative multiplicative function $\omega(d)$ satisfying
	$$
	\left\{
	\begin{array}{ll}
	\omega(1)=1, & \\ 
	0 \leq \omega(p)<p, & \text { if } p \notin B, \\ 
	\omega(p)=0, & \text { if } p \in B, 
	\end{array}\right. 
	$$
	and for some fixed (independent of $z, z_1$) constants $\kappa>1$ and $A \geq 2$,
	$$
	\prod_{2 \leq p<z}\left(1-\frac{\omega(p)}{p}\right)^{-1} \leq\left(\frac{\log z}{\log z_1}\right)^\kappa\left(1+\frac{A}{\log z_1}\right), 
	\quad {\rm for} \ 2 \leq z_1<z. 
	$$
	Write
	$$
	R_d=\left|\mathscr{A}_d\right|-\frac{\omega(d)}{d} X
	$$
	The quantity $\frac{\omega(d)}{d} X$ is considered as an approximation to $\left|\mathscr{A}_d\right|$, and therefore we suppose that the errors $R_d$ are small on average, in the sense that for some constant $\tau$ with $0<\tau<1, A_1 \geq 1$, and $A_2 \geq 2$,
	\begin{equation}\label{def/tau}
	\sum\limits_{d<X^\tau \log^{-A_1}X \atop (d,B)=1}\mu^{2}(d)4^{\nu(d)}|R_d|\le A_2\dfrac{X}{\log^{\kappa+1}X}
	\end{equation}
	where $\nu(d)$ denotes the number of prime factors of $d$. 
	This $\tau$ is called {\it level of equi-distributions}, and is a crucial constant in sieve methods. 
	Finally, we introduce a constant $\mu$ such that
	$$
	\max _{a_n \in \mathscr{A}} n \leq X^{\tau \mu}
	$$
	The following two lemmas are essentially in Diamond and Halberstam \cite[Theorems 0 and 1]{DiaHal}.
	
	\begin{lemma}
		Let $\kappa>1$ be given, and let $\sigma_\kappa(u)$ be the continuous solution of the differential-difference problem
		$$
		\left\{
		\begin{array}{ll}
		u^{-\kappa} \sigma(u)=A_\kappa^{-1}, \quad & 0<u \leq 2, \ A_\kappa=(2 e^\gamma)^\kappa \Gamma(\kappa+1),
		\\
		(u^{-\kappa} \sigma(u))^{\prime}=-\kappa u^{-\kappa-1} \sigma(u-2), \quad & 2<u.
		\end{array}
		\right.
		$$
		There exist two numbers $\alpha_{\kappa}$ and $\beta_{\kappa}$ satisfying
		$$
		\alpha_{\kappa} \geq \beta_{\kappa} \geq 2
		$$
		such that the simultaneous differential-difference system
		$$
		\left\{\begin{array}{ll}
		F(u)=1 / \sigma_{\kappa}(u), & 0<u \leq \alpha_{\kappa}, \\
		f(u)=0, & 0<u \leq \beta_\kappa, \\
		(u^\kappa F(u))^{\prime}=\kappa u^{\kappa-1} f(u-1), & u>\alpha_{\kappa}, \\
		(u^\kappa f(u))^{\prime}=\kappa u^{\kappa-1} F(u-1), & u>\beta_\kappa,
		\end{array}\right.
		$$
		has continuous solutions $F_\kappa(u)$ and $f_\kappa(u)$ with the properties that
		$$
		F_\kappa(u)=1+O(e^{-u}), \quad f_\kappa(u)=1+O(e^{-u}),
		$$
		and that $F_\kappa(u)$ and $f_\kappa(u)$, respectively, decreases and increases monotonically towards $1$ as $u \rightarrow \infty$. 
	\end{lemma}
	
	\begin{lemma}\label{Weight}
		Let $\mathscr{A}$ and $B$ be as above. Then, for any two real numbers $u$ and $v$ satisfying
		$$
		\frac{1}{\tau}<u \leq v, \quad \beta_\kappa<\tau v,
		$$
		we have
		$$
		\sum_{n \in P_r(B)} a_n \gg X \prod_{p<X^{1 / v}}\bigg(1-\frac{\omega(p)}{p}\bigg)
		$$
		provided only that
		$$
		r>\tau \mu u-1+\frac{\kappa}{f_\kappa(\tau v)} \int_1^{v / u} F_\kappa(\tau v-s)\bigg(1-\frac{u}{v} s\bigg) \frac{d s}{s} .
		$$
	\end{lemma}
	
	We remark that Lemma \ref{Weight} also holds in the case 
	$\kappa=1$, as is pointed out in Diamond and Halberstam \cite[Theorem 11.1]{DiaHal08}. 	
	
	\subsection{The linear case $\kappa=1$} 
	We are going to need the following two lemmas for the 
	linear sieve; for proof, see Halberstam, Heath-Brown,  and Richert 
	\cite[(3.11) and (3.12)]{HHR81}. 
	
	\begin{lemma}\label{F=/}
		For the function $F(s)$ in the linear sieve, we have
		$$
		F(s)
		=\left\{
		\begin{aligned}
		&F_1(s), \quad 1\le s \leq 3, \\
		&F_2(s), \quad 3 \leq s \leq 5, \\
		&F_3(s), \quad 5 \leq s \leq 7, 
		\end{aligned}
		\right. 
		$$
		with 
		$$
		F_1(s)=\frac{2 e^\gamma}{s}, \quad 
		F_2(s)=\frac{2 e^\gamma}{s}\bigg(1+\int_2^{s-1} \frac{\log (t-1)}{t} d t\bigg), \quad 
		$$
		$$
		F_3(s)=\frac{2 e^\gamma}{s}\bigg(1+\int_2^{s-1} \frac{\log (t-1)}{t} d t 
		+\int_2^{s-3} \frac{\log (t-1)}{t} d t \int_{t+2}^{s-1} \frac{1}{u} \log \frac{u-1}{t+1} d u\bigg), 
		$$
		where $\gamma=0.577 \ldots$ is the Euler constant. 
	\end{lemma}
	
	\begin{lemma}\label{f=/}
		For the function $f(s)$ in the linear sieve, we have
		$$ 
		f(s)
		=\left\{
		\begin{aligned}
		&f_1(s), \quad 2 \leq s \leq 4, \\
		&f_2(s), \quad 4 \leq s \leq 6, \\
		&f_3(s), \quad 6 \leq s \leq 8,
		\end{aligned}
		\right. 
		$$
		with 
		$$
		f_1(s)=\frac{2 e^\gamma}{s} \log (s-1), \quad 
		f_2(s)=\frac{2 e^\gamma}{s}\bigg(\log (s-1)+\int_3^{s-1} \frac{d t}{t} \int_2^{t-1} \frac{\log (u-1)}{u} d u\bigg), 
		$$
		\begin{eqnarray*}		
			\begin{aligned}
				f_3(s)=
				&\frac{2 e^\gamma}{s}\left(\log (s-1)+\int_3^{s-1} \frac{d t}{t} \int_2^{t-1} \frac{\log (u-1)}{u} d u\right. \\
				&\left.+\int_2^{s-4} \frac{\log (t-1)}{t} d t \int_{t+2}^{s-2} \frac{1}{u} \log \frac{u-1}{t+1} \log \frac{s}{u+2} d u\right), 
			\end{aligned}
		\end{eqnarray*}		
		where $\gamma=0.577 \ldots$ is the Euler constant. 
	\end{lemma}
	
	We note that 
	$$
	F(7)\leq 1.0000050, \quad f(8)\geq 0.9999648, 
	$$
	and these values are very close to $1$. This explains why in the above two lemmas we confine our iteration procedure 
	for $F$ to $s\leq 7$, and for $f$ to $s\leq 8$. Computations involving $F$ and $f$ for larger $s$ needs extra efforts, 
	but will not improve the final results. 
	
	\begin{proposition}\label{Prop35}
		Let $\mathscr{A}$ and $B$ be as above. For 
		\begin{eqnarray}\label{def/ab}		
		1\le a< 3< a+5< b\le 8, 
		\end{eqnarray}		
		we have
		\begin{equation}\label{nPrBabgg}
		\sum_{n \in P_r(B)} a_n \gg X \prod_{p<X^{\tau / b}}\bigg(1-\frac{\omega(p)}{p}\bigg)
		\end{equation}
		with 
		\begin{equation}\label{35r>}
		r>\dfrac{b}{(b-a)\tau}-1+\dfrac{2e^{\gamma}}{f(b)}(I_1+I_2+I_3),
		\end{equation}
		where
		$$
		\begin{aligned}
		&I_1=\dfrac{1}{b}\log\dfrac{(b-1)(b-a)}{a}-\dfrac{1}{b-a}\log\dfrac{b-1}{a},\\
		&I_2=\int_{3}^{b-1}\dfrac{1}{t}\bigg(\dfrac{1}{b-t}-\dfrac{1}{b-a}\bigg)dt\int_{2}^{t-1}\dfrac{\log(u-1)}{u}du,\\
		&I_3=\int_{5}^{b-1}\dfrac{1}{t}\bigg(\dfrac{1}{b-t}-\dfrac{1}{b-a}\bigg)dt\int_{2}^{t-3}\dfrac{\log(u-1)}{u}du\int_{u+2}^{t-1}\dfrac{1}{v}\log\dfrac{v-1}{u+1}dv.
		\end{aligned}
		$$
	\end{proposition}
	
	\begin{proof}
		We apply Lemma \ref{Weight} with $\kappa=1$ that is allowed by the remark immediately after the statement 
		Lemma \ref{Weight}. 
		Thus we have $\alpha_1=\beta_1=2, \tau\mu=1$. It follows that 
		for any two real numbers $u$ and $v$ satisfying
		\begin{equation}\label{check}
		\frac{1}{\tau}<u \leq v,\quad \tau v>2,
		\end{equation}
		we have
		$$
		\sum_{n \in P_r(B)} a_n \gg X \prod_{p<X^{1 / v}}\bigg(1-\frac{\omega(p)}{p}\bigg)
		$$
		provided only that
		$$
		r>u-1+\frac{1}{f(\tau v)} \int_1^{v / u} F(\tau v-s)\bigg(1-\frac{u}{v} s\bigg) \frac{d s}{s} .
		$$
		For $a, b$ satisfying \eqref{def/ab}, we let 
		$$
		u=\dfrac{b}{(b-a)\tau},\quad v=\dfrac{b}{\tau}, 
		$$
		so that \eqref{check} holds. Hence we have the estimate \eqref{nPrBabgg} provided that 
		\begin{equation}\label{r>2nd}
		r>\dfrac{b}{(b-a)\tau}-1+\frac{1}{f(b)} \int_1^{b-a} F(b-s)\bigg(\frac{1}{s}-\frac{1}{b-a}\bigg) ds .
		\end{equation}
		To compute the integral in \eqref{r>2nd}, we split it into three parts, so that 
		\begin{eqnarray*} 
			\int_1^{b-a} F(b-s)\bigg(\frac{1}{s}-\frac{1}{b-a}\bigg) ds
			=\int_1^{b-5} + \int_{b-5}^{b-3}+ \int_{b-3}^{b-a}
			= J_1+J_2+J_3, 
		\end{eqnarray*} 
		say. We apply the expressions for $F(s)$ in Lemma \ref{F=/} to compute the three integrals above, getting 
		$$
		\begin{aligned}
		&J_3=\int_{b-3}^{b-a} \frac{2 e^\gamma}{b-s}\bigg(\frac{1}{s}-\frac{1}{b-a}\bigg) ds,\\
		&J_2=\int_{b-5}^{b-3} \frac{2 e^\gamma}{b-s}\bigg(\frac{1}{s}-\frac{1}{b-a}\bigg) ds+2 e^\gamma\int_{3}^{5}\dfrac{1}{t}\bigg(\dfrac{1}{b-t}-\dfrac{1}{b-a}\bigg)dt\int_{2}^{t-1}\dfrac{\log(u-1)}{u}du, \\
		&J_1=\int_{1}^{b-5} \frac{2 e^\gamma}{b-s}\bigg(\frac{1}{s}-\frac{1}{b-a}\bigg) ds+2 e^\gamma\int_{5}^{b-1}\dfrac{1}{t}\bigg(\dfrac{1}{b-t}-\dfrac{1}{b-a}\bigg)dt\int_{2}^{t-1}\dfrac{\log(u-1)}{u}du\\
		&\quad +2 e^\gamma\int_{5}^{b-1}\dfrac{1}{t}\bigg(\dfrac{1}{b-t}-\dfrac{1}{b-a}\bigg)dt
		\int_{2}^{t-3}\dfrac{\log(u-1)}{u}du\int_{u+2}^{t-1}\dfrac{1}{v}\log\dfrac{v-1}{u+1}dv. 
		\end{aligned}
		$$
		Finally it is easy to check that 
		$$
		\int_{1}^{b-a} \frac{1}{b-s}\bigg(\frac{1}{s}-\frac{1}{b-a}\bigg) ds
		=\dfrac{1}{b}\log\dfrac{(b-1)(b-a)}{a}-\dfrac{1}{b-a}\log\dfrac{b-1}{a}. 
		$$
		Collecting everything back into \eqref{r>2nd}, 
		we get the assertion of the lemma.  
	\end{proof}
	
	\section{Proof of the Theorems}
	Finally we give the proof of the Theorems \ref{Thm1} and \ref{Thm2}.  
	
	\subsection{Proof of Theorem \ref{Thm1}}
	We apply Proposition \ref{Prop35}, and recall that now $\kappa=1$. The level $\tau$ in \eqref{def/tau} 
	in the present situation can be deduced from Lemma \ref{fO/dens}; in particular 
	\eqref{anTd} implies that 
	$$
	\tau=\frac{1}{4}-\frac{\theta}{2}=\frac{25}{128}. 
	$$ 
	Now we choose $a=1, b=6.6$ in \eqref{def/ab}, so that \eqref{nPrBabgg} and \eqref{35r>} now 
	take the form 
	\begin{equation}\label{Pr/Bab}
	\sum_{n \in P_r(B)} a_n \gg X \prod_{p<X^{\tau / b}}\bigg(1-\frac{\omega(p)}{p}\bigg)
	\end{equation}
	with 
	\begin{equation}\label{Thm11r>}
	r>\dfrac{1056}{175}-1+\dfrac{2e^{\gamma}}{f(6.6)}(I_1+I_2+I_3), 
	\end{equation}
	where
	$$
	\begin{aligned}
	&I_1=\dfrac{115}{924}\log\dfrac{28}{5}<0.21442,\\
	&I_2=\int_{3}^{5.6}\dfrac{1}{t}\bigg(\dfrac{1}{6.6-t}-\dfrac{1}{5.6}\bigg)dt\int_{2}^{t-1}\dfrac{\log(u-1)}{u}du<0.05558,\\
	&I_3=\int_{5}^{5.6}\dfrac{1}{t}\bigg(\dfrac{1}{6.6-t}-\dfrac{1}{5.6}\bigg)dt
	\int_{2}^{t-3}\dfrac{\log(u-1)}{u}du\int_{u+2}^{t-1}\dfrac{1}{v}\log\dfrac{v-1}{u+1}dv<0.00001.\\
	\end{aligned}
	$$
	By Lemma \ref{f=/} we also have 
	$$
	\dfrac{2e^{\gamma}}{f(6.6)}<3.5623.
	$$
	Therefore \eqref{Thm11r>} becomes 
	$r > 5.996$, and hence $r=6$ is acceptable. 
	With $r=6$, the above \eqref{Pr/Bab} becomes 
	\begin{equation}\label{Pr/Bab/+}
	\sum_{n \in P_6(B)} a_n \gg X \prod_{p<X^{3/100}}\bigg(1-\frac{\omega(p)}{p}\bigg)
	\gg \frac{X}{\log X}. 
	\end{equation}
	This proves the theorem unconditionally. 
	
	Now we handle the conditional case. Selberg's eigenvalue conjecture means that $\theta=0$, and hence 
	we have $\tau=1/4$. Choosing $a=1$ and $b=7$, we have 
	\begin{equation}\label{/Chm11r>}
	r>\dfrac{14}{3}-1+\dfrac{2e^{\gamma}}{f(7)}(I_1+I_2+I_3), 
	\end{equation}
	where
	$$
	\begin{aligned}
	&I_1=\dfrac{5}{42}\log 6<0.21331,\\
	&I_2=\int_{3}^{6}\dfrac{1}{t}\bigg(\dfrac{1}{7-t}-\dfrac{1}{6}\bigg)dt\int_{2}^{t-1}\dfrac{\log(u-1)}{u}du<0.07015,\\
	&I_3=\int_{5}^{6}\dfrac{1}{t}\bigg(\dfrac{1}{7-t}-\dfrac{1}{6}\bigg)dt\int_{2}^{t-3}\dfrac{\log(u-1)}{u}du
	\int_{u+2}^{t-1}\dfrac{1}{v}\log\dfrac{v-1}{u+1}dv<0.00003. 
	\end{aligned}
	$$
	By Lemma \ref{f=/} we also have 
	$$
	\frac{2e^{\gamma}}{f(7)}<3.5622.
	$$
	Therefore \eqref{/Chm11r>} becomes $r>4.676$. Hence $r=5$ is acceptable, and a formula like \eqref{Pr/Bab/+} 
	but with $r=6$ replaced by $r=5$ holds. This proves Theorem \ref{Thm1}. 
	
	\subsection{Proof of Theorem \ref{Thm2}} 
	The idea is similar to that of Theorem \ref{Thm1}, and so we may be brief and just indicated the differences.  
	The first difference is that in the present situation we have $\kappa=2$, and therefore we apply the non-linear sieves in 
	Lemma \ref{Weight} to prove Theorem \ref{Thm2}. 
	
	The second difference is that, instead of computing the functions 
	$F_{\kappa}(s)$ and $f_{\kappa}(s)$ for $s$ in various intervals, we use 
	the following estimate: For general $0<\tau\leq 1$
	and $\kappa>1$, and any $0<\zeta< \beta_\kappa$, put
	$$
	\tau u=1+\zeta-\f{\zeta}{\beta_\kappa}, \quad
	\tau v=\f{\beta_\kappa}{\zeta}+\beta_\kappa-1.
	$$
	Then
	$$
	\frac{\kappa}{f_{\kappa}(\tau v)}\int_{1}^{v/u}F_{\kappa}(\tau v-s)\bigg(1-\frac{u}{v}s\bigg)\dfrac{ds}{s}
	\le (\kappa+\zeta)\log \frac{\beta_{\kappa}}{\zeta}-\kappa+\zeta\dfrac{\kappa}{\beta_{\kappa}}. 
	$$
	This follows from Halberstam and Richert \cite[(10.1.10), (10.2.4), and (10.2.7)]{HR74}. 
	We remark that this approach has also been used in \cite{LS10}, and also in \cite{KO12} and \cite{KH15}. 
	
	Specifying $\kappa=2$, the above inequality becomes 
	$$ 
	\frac{2}{f_{2}(\tau v)}\int_{1}^{v/u}F_{2}(\tau v-s)
	\bigg(1-\dfrac{u}{v}s\bigg)\frac{ds}{s}\le (2+\zeta)\log\dfrac{\beta_2}{\zeta}-2+\zeta\dfrac{2}{\beta_2}. 
	$$
	Note that $\beta_2\approx 4.266450$ by \cite[Appendix III on p.345]{DiaHalRic}. Thus a sufficient 
	condition for Lemma \ref{Weight} is 
	\begin{equation}\label{Thm2r>}
	r>(1+\zeta)\mu-1+(2+\zeta)\log\dfrac{\beta_2}{\zeta}-2+\zeta\dfrac{2-\mu}{\beta_2}=:m(\zeta),
	\end{equation}
	say. Unconditionally we have $\tau=25/128$ as before, and therefore $\mu=2/\tau=10.24$. 
	It follows that 	
	$$
	\min_{0<\zeta<\beta_2} m(\zeta)=m(0.19214\cdots)\approx 15.6327.
	$$ 
	Assuming Selberg's eigenvalue conjecture, one has $\tau=1/4$ and therefore $\mu=8$. Hence 
	$$
	\min_{0<\zeta<\beta_2} m(\zeta)=m(0.23556\cdots)\approx 13.0287.
	$$
	In conclusion one can take $r=16$ or $14$ unconditionally or conditionally. This proves Theorem \ref{Thm2}. 
	
	\medskip 
	\noindent 
	{\bf Acknowledgements}. The authors would like to thank Wenjia Zhao for stimulating discussions, and the anonymous referee for helpful suggestions. 
	The authors are supported by the National Key Research and Development 
	Program of China (No. 2021YFA1000700), and the National Natural Science Foundation of China (Nos. 12031008 \& 123B2001).

\end{document}